\documentclass[10pt]{article}

 \usepackage{amsmath}
 \usepackage{amsthm,amssymb}
 \usepackage{makeidx,epsfig,lscape}
 \usepackage{color,colortbl}
 \usepackage{fancyhdr}
 \usepackage{times}
 \usepackage{xcolor,pict2e}
 \usepackage[latin1]{inputenc}
 
\newcommand{\R}{\mathbb R}
\newcommand{\N}{\mathbb N}

 \renewcommand{\headrulewidth}{0pt}
 \renewcommand{\footrulewidth}{0.5pt}

 \definecolor{myaqua}{rgb}{0.0,0.5,0.55}
 \definecolor{lightaqua}{rgb}{0.75,0.95,0.95}

 \usepackage[colorlinks = true,
            linkcolor = myaqua,
            urlcolor  = blue,
            citecolor = red,
            pdfpagemode=UseOutlines,
            bookmarksnumbered=true,
            pdfpagelabels,
            breaklinks]{hyperref}

\usepackage{caption}
\usepackage{floatrow}
\newtheorem{theorem}{Theorem}
\newtheorem{prop}{Proposition}
\newtheorem{lem}{Lemma}
\newtheorem{coro}{Corollary}
\newtheorem{defn}{Definition}[section]
\newtheorem{rem}{Remark}[section]
\def\lin#1#2{\textcolor[rgb]{0.6,0.6,0.6}{\vspace*{#1mm} \hrule
   height 3 pt \vspace*{#2mm}}}
\def\bt{\begin{tabular}}
\def\et{\end{tabular}}
\def\and{\mbox{ and }}

\def\P{\mbox{\bf P}}

\def\1{{\bf 1}}

 \def\boxx#1#2#3#4#5{
 {\linethickness{#4pt}\put(#1,#5){\color{myaqua}{\line(1,0){#3}}}}
 \multiput(#1,#2)(0,#4){2}{\line(1,0){#3}}
 \multiput(#1,#2)(#3,0){2}{\line(0,1){#4}}
  }

\begin{document}

% \fancyhead[L]{\hspace*{-13mm}
% \bt{l}{\bf Open Journal of *****, 2014, *,**}\\
% Published Online **** 2014 in SciRes.
% \href{http://www.scirp.org/journal/*****}{\color{blue}{\underline{\smash{http://www.scirp.org/journal/****}}}} \\
% \href{http://dx.doi.org/10.4236/****.2014.*****}{\color{blue}{\underline{\smash{http://dx.doi.org/10.4236/****.2014.*****}}}} \\
% \et}
% \fancyhead[R]{\special{psfile=pic1.ps hoffset=366 voffset=-33}}

 $\mbox{ }$

 \vskip 12mm

{ % \fontfamily{Cambria}\selectfont

% "Title of the Paper"
{\noindent{\Large\bf\color{myaqua}
  An ideal class to construct solutions for skew Brownian motion equations}} %\mathbb{R}_{+} $\Sigma(H)$
%
% \runtitle{Change-Point Analysis of Survival Data}
\\[6mm]
{\bf Fulgence EYI OBIANG$^1$, Octave MOUTSINGA$^2$ and Youssef OUKNINE$^3$}}
\\[2mm]
{ %\fontfamily{Calibri}\selectfont
$^1$URMI Laboratory, Département de Mathématiques et Informatique, Faculté des Sciences, Université des Sciences et Techniques de Masuku, BP: 943 Franceville, Gabon. 
\\
Email: \href{mailto:feyiobiang@yahoo.fr}{\color{blue}{\underline{\smash{feyiobiang@yahoo.fr}}}}\\[1mm]
$^2$ URMI Laboratory, Département de Mathématiques et Informatique, Faculté des Sciences, Université des Sciences et Techniques de Masuku, BP: 943 Franceville, Gabon. \\
\href{mailto:octavemoutsing-pro@yahoo.fr}{\color{blue}{\underline{\smash{octavemoutsing-pro@yahoo.fr}}}}\\[1mm]
$^3$LIBMA Laboratory, Department of Mathematics, Faculty of Sciences Semlalia, Cadi Ayyad University, P.B.O. 2390 Marrakech, Morocco.\\
 Hassan II Academy of Sciences and Technologies, Rabat, Morocco. \\
Africa Business School, Mohammed VI Polytechnic, Lot 660, HayMoulay Rachid, P.B.O. 43150, Benguerir, Moroco.\\
Email:\href{mailto:ouknine@ucam.ac.ma}{\color{blue}{\underline{\smash{ouknine@uca.ac.ma}}}}\\[1mm]
\lin{5}{7}

 {  %\fontfamily{Cambria}\selectfont
 {\noindent{\large\bf\color{myaqua} Abstract}{\bf \\[3mm]
 \textup{
 This paper contributes to the study of stochastic processes of the class $(\Sigma)$. First, we extend the notion of the above-mentioned class to càdlàg semi-martingales, whose finite variational part is considered càdlàg instead of continuous. Thus, we present some properties and propose a method to characterize such stochastic processes. Second, we investigate continuous processes of the class $(\Sigma)$. More precisely, we derive a series of new characterization results. In addition, we construct solutions for skew Brownian motion equations using continuous stochastic processes of the class $(\Sigma)$.
 }}}
 \\[4mm]
 {\noindent{\large\bf\color{myaqua} Keywords:}{\bf \\[3mm]
 Class $(\Sigma)$; Skew Brownian motion; Balayage formula; Honest time; Relative martingales.
}}\\[4mm]{\noindent{\large\bf\color{myaqua} MSC:}{\color{blue} 60G07; 60G20; 60G46; 60G48}}
\lin{3}{1}

\renewcommand{\headrulewidth}{0.5pt}
\renewcommand{\footrulewidth}{0pt}

 \pagestyle{fancy}
 \fancyfoot{}
 \fancyhead{} % clear all header and footer fields
 \fancyhf{}
 \fancyhead[RO]{\leavevmode \put(-140,0){\color{myaqua} Fulgence EYI OBIANG et al. (2020)} \boxx{15}{-10}{10}{50}{15} }
 %\fancyhead[LE]{\leavevmode \put(0,0){\color{myaqua}F. EYI-OBIANG et al (2015)}  \boxx{-45}{-10}{10}{50}{15} }
 \fancyfoot[C]{\leavevmode
 %\put(0,0){\color{lightaqua}\circle*{34}}
 %\put(0,0){\color{myaqua}\circle{34}}
 \put(-2.5,-3){\color{myaqua}\thepage}}

 \renewcommand{\headrule}{\hbox to\headwidth{\color{myaqua}\leaders\hrule height \headrulewidth\hfill}}
{\section*{Introduction}}

This study investigates semi-martingales of the class $(\Sigma)$. We consider stochastic processes $X$ of the form
$$X=M+A,$$
where $M$ is a local martingale and $A$ is an adapted finite variation process such that $dA_{t}$ is carried by $\{t\geq0:X_{t}=0\}$. Such processes are frequently encountered in stochastic analysis. Well-known examples of such processes include càdlàg local martingales, the absolute value of a continuous martingale, the positive and negative parts of a continuous martingale, solutions of skew Brownian motion equations starting from zero, and the drawdown of a càdlàg local martingale with only negative jumps. These processes play an important role in many probabilistic studies, e.g., the study of zeros of continuous martingales, the theory of Azéma--Yor martingales, the resolution of Skorokhod's reflection equation and embedding problem, and the study of Brownian local times.

The notion of the class $(\Sigma)$ has been studied extensively (see \cite{pat,eomt,naj,naj1,naj2,naj3,nik,mult,y1}), and its extensions have been presented in the literature. It was first introduced by Yor \cite{y1} for continuous positive submartingales and subsequently extended \cite{y} to continuous semi-martingales. Nikeghbali \cite{nik} and Cheridito et al. \cite{pat} proposed and studied an extension of this notion to càdlàg semi-martingales. However, we have highlighted some shortcomings of the above-mentioned developments, which we attempt to address in this study. First, in all the above-mentioned cases, the finite variational part of a process of the class $(\Sigma)$ is always considered continuous. Meanwhile, Nikeghbali presented two remarkable characterization results (Theorem 2.1 of \cite{nik} and Proposition 2.4 of \cite{mult}). The drawback of these results is that they only characterize positive submartingales of the class $(\Sigma)$. Finally, we focus on the construction of solutions of the following skew Brownian motion equations:
\begin{equation}
	dX_{t}=dB_{t}+(2\alpha-1)dL^{0}_{t}(X)
\end{equation}
and
\begin{equation}\label{ism}
	X_{t}=x+B_{t}+\int_{0}^{t}{(2\alpha(s)-1)dL_{s}^{0}(X)},
\end{equation}
 where $B$ is a standard Brownian motion, $x\in\R$, and $L_{t}^{0}(X)$ denotes the symmetric local time at $0$ of the unknown process $X$. Recall that these equations first appeared in the seminal work of Itô and Mckean \cite{11}; since then, they have been studied extensively \cite{siam,9,10,11,16,21}. In this study, we focus on the solution given by Bouhadou and Ouknine \cite{siam}, which is of the form
$$X_{t}=Z_{\gamma_{t}}|B_{t}|,$$
where $Z$ is a progressive process that we shall recall later, $B$ is a standard Brownian motion, and $\gamma_{t}=\sup\{s\leq t:B_{s}=0\}$. We remark that $|B|$ is an element of the class $(\Sigma)$. Furthermore, many results obtained on the processes of the class $(\Sigma)$ are generally extensions of the results initially proved for the Brownian motion. Hence, an intuitive question is to determine whether it is possible to construct solutions for the skew Brownian motion equations using other processes of the class $(\Sigma)$.

This study aims to contribute to the literature in the sense of the above-mentioned remarks. First, we present a general framework to study a larger class of càdlàg stochastic processes. More precisely, we propose extending the definition of Cheridito et al. \cite{pat} by weakening the continuity condition on the finite variational part of processes of the class $(\Sigma)$. In other words, we consider the following definition.
\begin{defn}\label{d1}
We say that a semi-martingale $X$ is of the class $(\Sigma)$ if it decomposes as $X=M+A$, where
\begin{enumerate}
	\item $M$ is a càdlàg local martingale, with $M_{0}=0$;
	\item $A$ is an adapted càdlàg predictable process with finite variations such that $A_{0-}=A_{0}=0$; 
	\item  $\int_{0}^{t}{1_{\{X_{s}\neq0\}}dA_{s}}=0$ for all $t\geq0$.
\end{enumerate}
\end{defn}
 Hence, we explore the general properties obtained for the previous versions of the class $(\Sigma)$ in \cite{nik,pat,mult}. For instance, we study the positive and negative parts of processes of the class $(\Sigma)$. We prove that the product of processes of the class $(\Sigma)$ with vanishing quadratic covariation is also of the class $(\Sigma)$. Further, we show that every positive process $X$ of the class $(\Sigma)$ admits a multiplicative decomposition. In other words, it can be decomposed as
$$X=CW-1,$$
where $W$ is a positive local martingale with $W_{0}=1$ and $C$ is a non-decreasing process. This result is an extension of that obtained by Nikeghbali for positive and continuous submartingales \cite{mult}. Finally, we generalize a result of Nikeghbali (Theorem 2.1 of \cite{nik}) that gives a martingale characterization for positive processes of the class $(\Sigma)$.

Second, we study continuous processes of the class $(\Sigma)$. To the best of our knowledge, this is the first study to present a series of results that permit characterization of all continuous processes (note necessary positive) of the class $(\Sigma)$. For instance, we extend the martingale characterization given in Theorem 2.1 of \cite{nik} as well as Proposition 2.4 of \cite{mult}. In addition, we obtain other characterization results using an interesting balayage formula given in Proposition2.2 of \cite{siam} and subsequently derive some corollaries. Finally, we focus on the construction of solutions for homogeneous and inhomogeneous skew Brownian Motion equations using continuous stochastic processes of the class $(\Sigma)$. More precisely, we generalize the construction of Bouhadou and Ouknine   to all continuous processes of the class $(\Sigma)$.

{\section{Processes of a new extension of the class \texorpdfstring{$(\Sigma)$}{sigma}}}

In this section, we present a framework to study stochastic processes satisfying the conditions of Definition \ref{d1}.

{\subsection{Preliminaries}}

Here, we explore some general properties of processes satisfying Definition \ref{d1}. Hence, we start by studying the positive and negative parts of processes of the class $(\Sigma)$ in the following lemma.

\begin{lem}\label{p3}
Let $X=M+A$ be a process of the class $(\Sigma)$. The following hold:
\begin{enumerate}
	\item If $A$ has no negative jump and $\int_{0}^{t}{1_{\{X_{s}\neq0\}}dA^{c}_{s}}=0$, then $X^{+}$ is a local submartingale.
	\item If $A$ has no positive jump and $\int_{0}^{t}{1_{\{X_{s}\neq0\}}dA^{c}_{s}}=0$, then $X^{-}$ is a local submartingale.
	\item If $X$ has no positive jump, then $X^{+}$ is also of the class $(\Sigma)$.
	\item If $X$ has no negative jump, then $X^{-}$ is also of the class $(\Sigma)$.
\end{enumerate}
\end{lem}
\begin{proof}
\begin{enumerate}
	\item From Tanaka's formula, we have
	$$X_{t}^{+}=\int_{0}^{t}{1_{\{X_{s^{-}}>0\}}dX_{s}}+\sum_{0<s\leq t}{1_{\{X_{s^{-}}\leq0\}}X_{s}^{+}}+\sum_{0<s\leq t}{1_{\{X_{s^{-}}>0\}}X_{s}^{-}}+\frac{1}{2}L_{t}^{0}.$$
	However,
	$$\hspace{-4cm}\int_{0}^{t}{1_{\{X_{s^{-}}>0\}}dX_{s}}=\int_{0}^{t}{1_{\{X_{s^{-}}>0\}}dM_{s}}+\int_{0}^{t}{1_{\{X_{s^{-}}>0\}}dA_{s}}$$
	$$\hspace{2cm}=\int_{0}^{t}{1_{\{X_{s^{-}}>0\}}dM_{s}}+\int_{0}^{t}{1_{\{X_{s^{-}}>0\}}dA^{c}_{s}}+\sum_{s\leq t}{1_{\{X_{s^{-}}>0\}}\Delta A_{s}}$$
	$$\hspace{2cm}=\int_{0}^{t}{1_{\{X_{s^{-}}>0\}}dM_{s}}+\int_{0}^{t}{1_{\{X_{s}>0\}}dA^{c}_{s}}+\sum_{s\leq t}{1_{\{X_{s^{-}}>0\}}\Delta A_{s}}$$ 
	since $A^{c}$ is continuous. Then,
	$$\int_{0}^{t}{1_{\{X_{s^{-}}>0\}}dX_{s}}=\int_{0}^{t}{1_{\{X_{s^{-}}>0\}}dM_{s}}+\sum_{s\leq t}{1_{\{X_{s^{-}}>0\}}\Delta A_{s}}$$
	because $dA^{c}$ is carried by $\{t\geq0:X_{t}=0\}$.
	Hence, we get
	$$X_{t}^{+}=\int_{0}^{t}{1_{\{X_{s^{-}}>0\}}dM_{s}}+\sum_{s\leq t}{1_{\{X_{s^{-}}>0\}}\Delta A_{s}}+\sum_{0<s\leq t}{1_{\{X_{s^{-}}\leq0\}}X_{s}^{+}}+\sum_{0<s\leq t}{1_{\{X_{s^{-}}>0\}}X_{s}^{-}}+\frac{1}{2}L_{t}^{0}.$$
	We know that $A$ has no negative jump. Thus, $(\sum_{s\leq t}{1_{\{X_{s^{-}}>0\}}\Delta A_{s}};t\geq0)$ is an increasing process that is null at zero. Moreover, $\left(\sum_{0<s\leq t}{1_{\{X_{s^{-}}\leq0\}}X_{s}^{+}}+\sum_{0<s\leq t}{1_{\{X_{s^{-}}>0\}}X_{s}^{-}}+\frac{1}{2}L_{t}^{0};t\geq0\right)$ is an increasing process that is vanishing at zero. Then, $X^{+}$ is a submartingale, since $M$ and $\int_{0}^{\cdot}{1_{\{X_{s^{-}}>0\}}dM_{s}}$ are local martingales.
	\item Now, we remark that $-X$ is also an element of the class $(\Sigma)$ and that its finite variational part, $-A$,  has no negative jump when $A$ has no positive jump. Therefore, it follows that $X^{-}=(-X)^{+}$ is a submartingale.
	\item We have
	$$X_{t}^{+}=\int_{0}^{t}{1_{\{X_{s^{-}}>0\}}dX_{s}}+\sum_{0<s\leq t}{1_{\{X_{s^{-}}>0\}}X_{s}^{-}}+\frac{1}{2}L_{t}^{0}$$
	since $X$ has no positive jump. Moreover,
	$$\int_{0}^{t}{1_{\{X_{s^{-}}>0\}}dX_{s}}=\int_{0}^{t}{1_{\{X_{s^{-}}>0\}}dM_{s}}+\int_{0}^{t}{1_{\{X_{s^{-}}>0\}}dA_{s}}.$$
  Hence,
	\begin{equation}\label{*}
	X_{t}^{+}=\int_{0}^{t}{1_{\{X_{s^{-}}>0\}}dM_{s}}+\int_{0}^{t}{1_{\{X_{s^{-}}>0\}}dA_{s}}+\sum_{0<s\leq t}{1_{\{X_{s^{-}}>0\}}X_{s}^{-}}+\frac{1}{2}L_{t}^{0}.
	\end{equation}
Now, let us set $Z_{t}=\sum_{0<s\leq t}{1_{\{X_{s^{-}}>0\}}X_{s}^{-}}$. Since $M$ and $\int_{0}^{\cdot}{1_{\{X_{s^{-}}>0\}}dM_{s}}$ are local martingales and $A$ is càdlàg, there exists a sequence of stopping times $(T_{n};n\in\N)$ increasing to $\infty$ such that 
	$$E[(X_{T_{n}})^{+}]=E[(M_{T_{n}}+A_{T_{n}})^{+}]<\infty\text{ and }E\left[\int_{0}^{T_{n}}{1_{\{X_{s^{-}}>0\}}dM_{s}}\right]=0\text{, }n\in\N.$$
	It follows from \eqref{*} that $E[Z_{T_{n}}]\leq E[(X_{T_{n}})^{+}]<\infty$ for all $n\in\N$. Thus, by Theorem VI.80 of \cite{pot}, there exists a right continuous increasing predictable process $V^{Z}$ such that $Z-V^{Z}$ is a local martingale vanishing at zero. Moreover, there exists a sequence of stopping times $(R_{n};n\in\N)$ increasing to $\infty$ such that
	$$E\left[\int_{0}^{t\wedge R_{n}}{1_{\{X^{+}_{s}\neq0\}}dV_{s}^{Z}}\right]=E\left[\int_{0}^{t\wedge R_{n}}{1_{\{X^{+}_{s}\neq0\}}d(V_{s}^{Z}-Z_{s})}+\int_{0}^{t\wedge R_{n}}{1_{\{X^{+}_{s}\neq0\}}dZ_{s}}\right]$$
	$$\hspace{-0.5cm}=E\left[\int_{0}^{t\wedge R_{n}}{1_{\{X^{+}_{s}\neq0\}}dZ_{s}}\right].$$
	Hence,
	$$E\left[\int_{0}^{t\wedge R_{n}}{1_{\{X^{+}_{s}\neq0\}}dV_{s}^{Z}}\right]=E\left[\sum_{0<s\leq t\wedge R_{n}}{1_{\{X_{s}^{+}\neq0\}}1_{\{X_{s^{-}}>0\}}X_{s}^{-}}\right]=E\left[\sum_{0<s\leq t\wedge R_{n}}{1_{\{X_{s}>0\}}1_{\{X_{s^{-}}>0\}}X_{s}^{-}}\right].$$
	Thus,
	$$E\left[\int_{0}^{t\wedge R_{n}}{1_{\{X^{+}_{s}\neq0\}}dV_{s}^{Z}}\right]=0,$$
	since $1_{\{X_{s}>0\}}X_{s}^{-}=0$. This implies that $\int_{0}^{t}{1_{\{X_{s}^{+}\neq0\}}dV_{s}^{Z}}=0$. Therefore, $dV_{t}^{Z}$ is carried by $\{t\geq0;X_{t}^{+}=0\}$. Consequently,
	$$X^{+}_{t}=\left(\int_{0}^{t}{1_{\{X_{s^{-}}>0\}}dM_{s}}+(Z_{t}-V_{t}^{Z})\right)+\left(V_{t}^{Z}+\int_{0}^{t}{1_{\{X_{s^{-}}>0\}}dA_{s}}+\frac{1}{2}L_{t}^{0}\right)$$
	is a stochastic process of the class $(\Sigma)$.
	\item It is obvious that $(-X)$ is of the class $(\Sigma)$ and it has no positive jump. Then, from 3), $X^{-}=(-X)^{+}$ is also of the class $(\Sigma)$.
\end{enumerate}
\end{proof}

%\begin{rem}
%We will prove in Corollary \ref{co5} of Section \ref{s3} that for any process $X$ of class $(\Sigma)$, $|X|$ is a submartingale. Which means that every positive %process of class $(\Sigma)$ is a submartingale.
%\end{rem}

Now, we shall show that the product of processes of the class $(\Sigma)$ with vanishing quadratic covariation is also of the class $(\Sigma)$.

\begin{lem}
Let $(X_{t}^{1})_{t\geq0},\cdots,(X_{t}^{n})_{t\geq0}$ be processes of the class $(\Sigma)$ such that $[ X^{i},X^{j}]=0$ for $i\neq j$. Then, $(\Pi_{i=1}^{n}{X^{i}_{t}})_{t\geq0}$ is also of the class $(\Sigma)$.
\end{lem}
\begin{proof}
Since $[ X^{1},X^{2}]=0$, integration by parts yields
$$X^{1}_{t}X^{2}_{t}=\int_{0}^{t}{X_{s-}^{1}dX_{s}^{2}}+\int_{0}^{t}{X_{s-}^{2}dX_{s}^{1}},$$
i.e.,
$$X^{1}_{t}X^{2}_{t}=\left[\int_{0}^{t}{X_{s-}^{1}dM_{s}^{2}}+\int_{0}^{t}{X_{s-}^{2}dM_{s}^{1}}\right]+\left[\int_{0}^{t}{X_{s-}^{1}dA_{s}^{2}}+\int_{0}^{t}{X_{s-}^{2}dA_{s}^{1}}\right].$$
It is easy to see that $M_{t}=\int_{0}^{t}{X_{s-}^{1}dM_{s}^{2}}+\int_{0}^{t}{X_{s-}^{2}dM_{s}^{1}}$ is a càdlàg local martingale. Furthermore, the process $A_{t}=\int_{0}^{t}{X_{s-}^{1}dA_{s}^{2}}+\int_{0}^{t}{X_{s-}^{2}dA_{s}^{1}}$ is a finite variation process such that
$$dA_{t}=X_{t-}^{1}dA_{t}^{2}+X_{t-}^{2}dA_{t}^{1}$$
is carried by $\{t\geq0:X^{1}_{t}X^{2}_{t}=0\}$. Therefore, $X^{1}X^{2}$ is of the class $(\Sigma)$. If $n\geq3$, then $[ X^{1}X^{2},X^{3}]=0$, and we obtain the result by induction.
\end{proof}

In the next lemma, we derive a new property using the balayage formula for càdlàg semi-martingales.
\begin{lem}
Let $X=M+A$ be a process of the class $(\Sigma)$ and denote $\gamma_{t}=\sup\{s\leq t:X_{s}=0\}$. Then, for any bounded predictable process $K$, $K_{\gamma_{\cdot}}X$ is an element of the class $(\Sigma)$ and its finite variational part is given by $\int_{0}^{\cdot}{K_{s}dA_{s}}$.
\end{lem}

\begin{proof}
We obtain the following by applying the balayage formula to the càdlàg case:
$$K_{\gamma_{t}}X_{t}=K_{\gamma_{0}}X_{0}+\int_{0}^{t}{K_{\gamma_{s}}dX_{s}}=\int_{0}^{t}{K_{\gamma_{s}}dM_{s}}+\int_{0}^{t}{K_{\gamma_{s}}dA_{s}}.$$
Since $dA_{t}$ is carried by $\{t\geq0:X_{t}=0\}$, we have the identity $K_{\gamma_{s}}dA_{s}=K_{s}dA_{s}$. Therefore,
$$K_{\gamma_{t}}X_{t}=\int_{0}^{t}{K_{\gamma_{s}}dM_{s}}+\int_{0}^{t}{K_{s}dA_{s}}.$$
It is easy to see that $\int_{0}^{\cdot}{K_{\gamma_{s}}dM_{s}}$ is a local martingale and that $K_{t}dA_{t}$ is carried by $\{t\geq0:K_{\gamma_{t}}X_{t}=0\}$. This completes the proof.
\end{proof}

\begin{coro}\label{c4}
Let $X=M+A$ be a process of the class $(\Sigma)$ and $f$ be a bounded Borel function. Then, the process 
$$\left(f(A_{t})X_{t}-F(A_{t}):t\geq0\right)$$ 
is a local martingale, where $F(A_{t})=\int_{0}^{t}{f(A_{s})dA_{s}}$.
\end{coro}

In Proposition 2.1 of \cite{mult}, Nikeghbali showed that every continuous non-negative local submartingale $Y$ with $Y_{0}=0$ decomposes as
$$Y=MC-1,$$
where $M$ is a continuous non-negative local martingale with $M_{0}=1$ and $C$ is a continuous increasing process with $C_{0}=1$. We extend this result to positive stochastic processes satisfying Definition \ref{d1} in the following corollary.
\begin{coro}\label{c5}
Let $X=M+A$ be a non-negative process of the class $(\Sigma)$. Then, there exist a càdlàg non-decreasing process $C$ and a càdlàg positive local martingale $W$ with $W_{0}=1$ such that $\forall t\geq0$,
$$X_{t}=C_{t}W_{t}-1.$$
\end{coro}
\begin{proof}
According to Lemma \ref{p3}, $X$ is a submartingale and $A$ is a non-decreasing process. Since the function $f$ defined by $f(x)=e^{-x}$ is a bounded Borel function on $[0,+\infty[$, it follows from Corollary \ref{c4} that 
$$\left(e^{-A_{t}}(X_{t}+1)-1:t\geq0\right)$$
 is a càdlàg local martingale that is null at zero. Then, 
$$W=\left(e^{-A_{t}}(X_{t}+1):t\geq0\right)$$
is a positive local martingale with $W_{0}=1$. Therefore, taking $C_{t}=e^{A_{t}}$, we obtain that $\forall t\geq0$,
$$X_{t}=C_{t}W_{t}-1.$$
This completes the proof.
\end{proof}

{\subsection{Extension of the martingale characterization}}

In this subsection, we generalize some known results subsequent to the martingale characterization of processes of the class $(\Sigma)$. Let us begin with those of Lemma 2.3 of \cite{pat}.
\begin{theorem}\label{t2}
Let $X=M+A$ be a process of the class $(\Sigma)$ and $A^{c}$ be the continuous part of $A$. For every $\mathcal{C}^{1}$ function $f$ and a function $F$ defined by $F(x)=\int_{0}^{x}{f(z)dz}$, the process 
$$\left(F(A^{c}_{t})-f(A^{c}_{t})X_{t}+\sum_{s\leq t}{[f(A^{c}_{s})-f^{'}(A^{c}_{s})X_{s}]\Delta A_{s}};t\geq0\right)$$
is a càdlàg local martingale.
\end{theorem}
\begin{proof}
Through integration by parts, we get
$$f(A^{c}_{t})X_{t}=\int_{0}^{t}{f(A^{c}_{s})dX_{s}}+\int_{0}^{t}{f^{'}(A^{c}_{s})X_{s-}dA^{c}_{s}}.$$
However, we have  
$$\int_{0}^{t}{f^{'}(A^{c}_{s})X_{s-}dA^{c}_{s}}=\int_{0}^{t}{f^{'}(A^{c}_{s})X_{s}dA^{c}_{s}}$$
since $A^{c}$ is a continuous process. Hence,
$$f(A^{c}_{t})X_{t}=\int_{0}^{t}{f(A^{c}_{s})dX_{s}}+\int_{0}^{t}{f^{'}(A^{c}_{s})X_{s}dA^{c}_{s}},$$
i.e.,
$$f(A^{c}_{t})X_{t}=\int_{0}^{t}{f(A^{c}_{s})dX_{s}}+\int_{0}^{t}{f^{'}(A^{c}_{s})X_{s}dA_{s}}-\sum_{s\leq t}{f^{'}(A^{c}_{t})X_{s}\Delta A_{s}}$$
because $A=A^{c}+\sum_{s\leq t}{\Delta A_{s}}$. Furthermore, we have $\int_{0}^{t}{f^{'}(A^{c}_{s})X_{s}dA_{s}}=0$ since $dA$ is carried by $\{t\geq0:X_{t}=0\}$. Therefore, it follows that
$$f(A^{c}_{t})X_{t}=\int_{0}^{t}{f(A^{c}_{s})dX_{s}}-\sum_{s\leq t}{f^{'}(A^{c}_{t})X_{s}\Delta A_{s}}$$
$$\hspace{2cm}=\int_{0}^{t}{f(A^{c}_{s})dM_{s}}+\int_{0}^{t}{f(A^{c}_{s})dA^{c}_{s}}+\sum_{s\leq t}{[f(A^{c}_{s})-f^{'}(A^{c}_{s})X_{s}]\Delta A_{s}}.$$
Consequently,
$$f(A^{c}_{t})X_{t}=\int_{0}^{t}{f(A^{c}_{s})dM_{s}}+F(A^{c}_{t})+\sum_{s\leq t}{[f(A^{c}_{s})-f^{'}(A^{c}_{s})X_{s}]\Delta A_{s}}.$$
This implies that
$$F(A^{c}_{t})+\sum_{s\leq t}{[f(A^{c}_{s})-f^{'}(A^{c}_{s})X_{s}]\Delta A_{s}}-f(A^{c}_{t})X_{t}=-\int_{0}^{t}{f(A^{c}_{s})dM_{s}}.$$
This completes the proof.
\end{proof}

\begin{coro}\label{c6}
Let $X=M+A$ be a process of the class $(\Sigma)$ such that the continuous part $A^{c}$ of the process $A$ satisfies the following: $\forall t\geq0$, 
$\int_{0}^{t}{1_{\{X_{s}\neq0\}}dA^{c}_{s}}=0$. Then, $f(A^{c})X$ is also of the class $(\Sigma)$. Further, its  finite variational part is given by
$$V_{t}=F(A^{c}_{t})+\sum_{s\leq t}{f(A^{c}_{s})\Delta A_{s}}.$$
\end{coro}
\begin{proof}
According to Theorem \ref{t2}, the process $W$ defined by $\forall t\geq0$,
$$W_{t}=F(A^{c}_{t})-f(A^{c}_{t})X_{t}+\sum_{s\leq t}{[f(A^{c}_{s})-f^{'}(A^{c}_{s})X_{s}]\Delta A_{s}}$$
is a càdlàg local martingale. Since $dA$ is carried by $\{t\geq0:X_{t}=0\}$, we get
$$\int_{0}^{t}{X_{s}dA_{s}}=\int_{0}^{t}{X_{s}dA^{c}_{s}}+\sum_{s\leq t}{X_{s}\Delta A_{s}}=0,$$
i.e., $\sum_{s\leq t}{X_{s}\Delta A_{s}}=0$ since $\int_{0}^{t}{X_{s}dA^{c}_{s}}=0$. Thus,
$$\sum_{s\leq t}{f^{'}(A^{c}_{s})X_{s}\Delta A_{s}}=0.$$
Therefore,
$$W_{t}=F(A^{c}_{t})-f(A^{c}_{t})X_{t}+\sum_{s\leq t}{f(A^{c}_{s})\Delta A_{s}}.$$
Consequently,
$$f(A^{c}_{t})X_{t}=-W_{t}+F(A^{c}_{t})+\sum_{s\leq t}{f(A^{c}_{s})\Delta A_{s}}.$$
This gives the result.
\end{proof}

\begin{rem}
Theorem \ref{t2} and Corollary \ref{c6} are natural extensions of Lemma 2.3 obtained by Cheridito et al. \cite{pat} for continuous $A$.
\end{rem}

Now, we shall present an extension of the martingale characterization to non-negative submartingales (Theorem 2.1 of \cite{nik}).

\begin{theorem}\label{t3}
Let $X=M+A$ be a positive semi-martingale. Then, the following are equivalent:
\begin{enumerate}
  \item $X\in(\Sigma)$;
	\item There exists a non-decreasing predictable process $V$ such that for any $F\in C^2$, the process 
	$$\left(F(V^{c}_{t})-F^{'}(V^{c}_{t})X_{t}+\sum_{s\leq t}{[F^{'}(V^{c}_{s})-F^{''}(V^{c}_{s})X_{s}]\Delta V_{s}};t\geq0\right)$$
is a càdlàg local martingale and $V\equiv A$.
\end{enumerate}
\end{theorem}
\begin{proof}
$(1)\Rightarrow(2)$ Let us take $V=A$. Hence, from Theorem \ref{t2}, we get that
$$\left(F(A^{c}_{t})-F^{'}(A^{c}_{t})X_{t}+\sum_{s\leq t}{[F^{'}(A^{c}_{s})-F^{''}(A^{c}_{s})X_{s}]\Delta A_{s}};t\geq0\right)$$
is a càdlàg local martingale.\\
$(2)\Rightarrow(1)$ First, let us take $F(x)=x$. Then, the process $W$ defined by 
$$W_{t}=V^{c}_{t}+\sum_{s\leq t}{\Delta V_{s}}-X_{t}=V_{t}-X_{t}$$
is a local martingale. Hence, by the uniqueness of the Doob--Meyer decomposition, we get $V=A$. Next, we take $F(x)=x^{2}$. Then, the process $B$ defined by 
$$B_{t}=(V_{t}^{c})^{2}-2V_{t}^{c}X_{t}+2\sum_{s\leq t}{V_{s}^{c}\Delta V_{s}}-2\sum_{s\leq t}{X_{s}\Delta V_{s}}$$
is a local martingale. However, through integration by part, it follows that
$$B_{t}=2\int_{0}^{t}{V^{c}_{s}dV^{c}_{s}}-2\int_{0}^{t}{V^{c}_{s}dX_{s}}-2\int_{0}^{t}{X_{s}dV^{c}_{s}}+2\sum_{s\leq t}{V_{s}^{c}\Delta V_{s}}-2\sum_{s\leq t}{X_{s}\Delta V_{s}}$$
$$\hspace{-0.75cm}=2\int_{0}^{t}{V^{c}_{s}d\left(V^{c}_{s}+\sum_{u\leq s}{\Delta V_{u}}-X_{s}\right)}-2\int_{0}^{t}{X_{s}d\left(V^{c}_{s}+\sum_{u\leq s}{\Delta V_{u}}\right)}$$
$$\hspace{-6cm}=2\int_{0}^{t}{V^{c}_{s}dW_{s}}-2\int_{0}^{t}{X_{s}dV_{s}}.$$
Consequently, we must have
$$\int_{0}^{t}{X_{s}dV_{s}}=0.$$
In other words, $dA$ is carried  by the set $\{t\geq0:X_{t}=0\}$.
\end{proof}

{\section{Contribution to continuous semi-martingales of the class \texorpdfstring{$(\Sigma)$}{sigma}}\label{s3}}

Now, we study continuous processes of the class $(\Sigma)$. First, we state some new characterization results. Second, we construct solutions for skew Brownian motion equations. A well-known continuous process of the class $(\Sigma)$ is the absolute value of a Brownian motion $|B|$. Bouhadou and Ouknine \cite{siam} constructed a solution for the inhomogeneous skew Brownian motion equation using the process $|B|$. Our contribution to this topic is to extend the construction of Bouhadou and Ouknine to all continuous processes of the class $(\Sigma)$.

For the readers' benefit, we first recall some useful results and terminologies.

\subsection{Recalling useful results}
We begin by defining two stochastic processes that are important for the present study. First, we remark that for any continuous semi-martingale $Y$, the set $\mathcal{W}=\{t\geq0; Y_{t}=0\}$ cannot be ordered. However, the set $\R_{+}\setminus\mathcal{W}$ can be decomposed as a countable union $\cup_{n\N}{J_{n}}$ of intervals $J_{n}$. Each interval $J_{n}$ corresponds to some excursion of $Y$. In other words, if $J_{n}=]g_{n},d_{n}[$, $Y_{t}\neq0$ for all $t\in]g_{n},d_{n}[$ and $Y_{g_{n}}=Y_{d_{n}}=0$. For any constant $\alpha\in[0,1]$, we consider a sequence $(\xi_n)$ of i.i.d. Bernoulli variables such that
$$P(\zeta_{n}=1)=\alpha\text{ and }P(\zeta_{n}=-1)=1-\alpha.$$
Now, let us define the process $Z^{\alpha}$ as follows.
\begin{equation}\label{zalpha}
	Z^{\alpha}_{t}=\sum_{n=0}^{+\infty}{\zeta_{n}1_{]g_{n},d_{n}[}(t)}.
\end{equation}

If we assume that $\alpha$ is a piecewise constant function associated with a partition $(0=t_{0}<t_{1}<\cdots<t_{n-1}<t_{m})$, i.e., $\alpha$ is of the form
$$\alpha(t)=\sum_{i=0}^{m}{\alpha_{i}1_{[t_{i},t_{i+1})}(t)},$$
where $\alpha_{i}\in[0,1]$ for all $i=0,1,\cdots,m$, then we shall consider the process
\begin{equation}\label{Zalpha}
	\mathcal{Z}^{\alpha}_{t}=\sum_{n=0}^{+\infty}{\sum_{i=0}^{m}{\zeta^{i}_{n}1_{]g_{n},d_{n}[\cap[t-{i},t_{i+1})}(t)}},
\end{equation}
where $(\zeta^{i}_{n})_{n\geq0}$, $i=1,2,\cdots,m$, are $m$ independent sequences of independent variables such that
$$\P(\zeta^{i}_{n}=1)=\alpha_{i}\text{ and }\P(\zeta^{i}_{n}=-1)=1-\alpha_{i}.$$

The balayage formula, especially the balayage formula for continuous semi-martingales in the progressive case, is a critical tool in this study. We recall it below.

\begin{prop}\label{balpro}
Let $Y$ be a continuous semi-martingale and $\gamma^{'}_{t}=\sup\{s\leq t:Y_{s}=0\}$. Let $k$ be a bounded progressive process, where ${^{p}k_{\cdot}}$ denotes its predictable projection. Then,
$$k_{\gamma^{'}_{t}}Y_{t}=k_{0}Y_{0}+{\int_{0}^{t}{^{p}k_{\gamma^{'}_{s}}dY_{s}}+R_{t}},$$
where $R$ is an adapted, continuous process with bounded variations such that $dR_{t}$ is carried by the set $\{Y_{s}=0\}$.
\end{prop}

Proposition \ref{balpro} is a powerful and interesting tool. However, the fact that we know nothing about the form of the process $R$ can be limiting. The processes $Z^{\alpha}$ and $\mathcal{Z}^{\alpha}$ are critical to this study. Bouhadou and Ouknine \cite{siam} identified the process $R$ of Proposition \ref{balpro} when the progressive process $k$ is equal to $Z^{\alpha}$ or $\mathcal{Z}^{\alpha}$. We recall these results below.  

\begin{prop}[\textbf{Ouknine and Bouhadou \cite{siam}}]\label{pzalph}
Let $Y$ be a continuous semi-martingale and $Z^{\alpha}$ be the process defined in \eqref{zalpha}. Then,
$$Z^{\alpha}_{t}Y_{t}=\int_{0}^{t}{Z^{\alpha}_{s}dY_{s}}+(2\alpha-1)L_{t}^{0}(Z^{\alpha}Y),$$
where $L_{\cdot}^{0}(Z^{\alpha}Y)$ is the local time of the semi-martingale $Z^{\alpha}Y$.
\end{prop}

\begin{prop}[\textbf{Ouknine and Bouhadou \cite{siam}}]
Let $Y$ be a continuous semi-martingale and $Z^{\alpha}$ be the process defined in \eqref{zalpha}. Then,
$$\mathcal{Z}^{\alpha}_{t}Y_{t}=\int_{0}^{t}{\mathcal{Z}^{\alpha}_{s}dY_{s}}+\int_{0}^{t}{(2\alpha(s)-1)dL_{s}^{0}(\mathcal{Z}^{\alpha}Y)},$$
where $L_{\cdot}^{0}(\mathcal{Z}^{\alpha}Y)$ is the local time of the semi-martingale $\mathcal{Z}^{\alpha}Y$.
\end{prop}

We conclude this subsection by recalling an important theorem of \cite{eomt}, i.e., a result that enables us to characterize stochastic processes of the class $(\Sigma)$. 
 \begin{theorem}\label{abs}
 Let $X$ be a continuous process that vanishes at zero. Then,
 $$X\in(\Sigma) \Leftrightarrow |X|\in(\Sigma).$$
 \end{theorem}

{\subsection{New characterization results for continuous semi-martingales of the class \texorpdfstring{$(\Sigma)$}{sigma}}}

Now, we shall state new characterization results for all continuous processes of the class $(\Sigma)$. We begin by extending Theorem 2.1 of \cite{nik}, which characterizes only non-negative submartingales of the class $(\Sigma)$.

\begin{prop}
Let $X$ be a continuous semi-martingale. The following are equivalent:
\begin{enumerate}
	\item $X\in(\Sigma)$;
	\item For every locally bounded Borel function $f$, the process 
	$$\left(f(L_{t})|X_{t}|-F(L_{t});t\geq0\right)$$
	is a local martingale, where $L$ is the local time of $X$ at level zero and $F(x)=\int_{0}^{x}{f(z)dz}$.
\end{enumerate}
\end{prop}
\begin{proof}
According to Theorem \ref{abs}, we have $X\in(\Sigma)$ if, and only if $|X|$ is a non-negative submartingale of the class $(\Sigma)$. However, from the martingale characterization of Nikeghbali \cite{nik}, this is equivalent to the fact that the process
$$\left(f(L_{t})|X_{t}|-F(L_{t});t\geq0\right)$$
	is a local martingale. This completes the proof.
\end{proof}

The following proposition extends another result characterizing the positive submartingales of the class $(\Sigma)$ (see Proposition 2.4 of \cite{mult}).

\begin{prop}
Let $X$ be a continuous semi-martingale. The following are equivalent:
\begin{enumerate}
	\item $X\in(\Sigma)$;
	\item There exists a unique strictly positive, continuous local martingale $M$, with $M_{0}=1$, such that
	$$|X_{t}|=\frac{M_{t}}{I_{t}}-1,$$
	where $$I_{t}=\inf_{s\leq t}{M_{s}}.$$
\end{enumerate}
The local martingale $M$ is given by
$$M_{t}=(1+|X_{t}|)\exp{(-L_{t})}.$$
\end{prop}

Next, we present a new method to characterize stochastic processes of the class $(\Sigma)$ using the progressive process $Z^{\alpha}$ defined in \eqref{zalpha}.
\begin{theorem}\label{z}
Let $X$ be a continuous semi-martingale. The following are equivalent:
\begin{enumerate}
	\item $X\in(\Sigma)$.
	\item $\forall\alpha\in[0,1]$, $Z^{\alpha}X\in(\Sigma)$.
	\item $\exists\alpha\in[0,1]$ such that $Z^{\alpha}X\in(\Sigma)$.
\end{enumerate}
\end{theorem}
\begin{proof}
$1\Rightarrow2)$ Let $X=M+V$ be an element of the class $(\Sigma)$. From Proposition 2.2 of \cite{siam}, we have
$$Z^{\alpha}_{t}X_{t}=\int_{0}^{t}{Z^{\alpha}_{s}dX_{s}}+(2\alpha-1)L_{t}^{0}(Z^{\alpha}X)$$
$$\hspace{3cm}=\int_{0}^{t}{Z^{\alpha}_{s}dM_{s}}+\int_{0}^{t}{Z^{\alpha}_{s}dV_{s}}+(2\alpha-1)L_{t}^{0}(Z^{\alpha}X).$$
However, we know that $\int_{0}^{t}{Z^{\alpha}_{s}dV_{s}}=0$ since $dV_{t}$ is carried by $\{t\geq0; X_{t}=0\}$ and $X_{t}=0$ $\Leftrightarrow$ $Z^{\alpha}_{t}=0$. Hence,
\begin{equation}\label{dem}
Z^{\alpha}_{t}X_{t}=\int_{0}^{t}{Z^{\alpha}_{s}dM_{s}}+(2\alpha-1)L_{t}^{0}(Z^{\alpha}X).	
\end{equation}
Then, $Z^{\alpha}X\in(\Sigma)$ since $(2\alpha-1)dL_{t}^{0}(Z^{\alpha}X)$ is carried by $\{t\geq0; Z^{\alpha}_{t}X_{t}=0\}$ and $\left(\int_{0}^{t}{Z^{\alpha}_{s}dM_{s}};t\geq0\right)$ is a local martingale.\\
$2\Rightarrow3)$ If we assume that $\forall\alpha\in[0,1]$, $Z^{\alpha}X\in(\Sigma)$. In particular, it follows that $\exists\alpha\in[0,1]$ such that $Z^{\alpha}X\in(\Sigma)$.\\
$3\Rightarrow1)$ Now, assume that $\exists\alpha\in[0,1]$ such that $Z^{\alpha}X\in(\Sigma)$. Then, according to Theorem \ref{abs}, $|Z^{\alpha}X|\in(\Sigma)$. However, $\forall t\geq0$, $Z^{\alpha}_{t}\in\{-1,0,1\}$ and $Z^{\alpha}_{t}=0\Leftrightarrow X_{t}=0$. Therefore,  
$$|Z^{\alpha}X|=|X|.$$
Consequently, 
$$|X|\in(\Sigma).$$
This completes the proof.
\end{proof}

\begin{rem}\label{o}
In the above-mentioned theorem, we have proved that when $X\in(\Sigma)$, we have $\forall\alpha\in[0,1]$,
$$Z^{\alpha}_{t}X_{t}=\int_{0}^{t}{Z^{\alpha}_{s}dM_{s}}+(2\alpha-1)L_{t}^{0}(Z^{\alpha}X).$$
\end{rem}

Now, as an application of Theorem \ref{z}, we have the following corollary, which gives a new martingale characterization of the class $(\Sigma)$.
\begin{coro}\label{cmart}
Let $X$ be a continuous semi-martingale. The following holds:
$$X\in(\Sigma)\Leftrightarrow \exists\alpha\in[0,1] \text{ such that }Z^{\alpha}X \text{ is a local martingale}.$$
\end{coro}
\begin{proof}
$\Rightarrow)$ It follows from Remark \ref{o} that $\forall \alpha\in[0,1]$,
$$Z^{\alpha}_{t}X_{t}=\int_{0}^{t}{Z^{\alpha}_{s}dM_{s}}+(2\alpha-1)L_{t}^{0}(Z^{\alpha}X).$$
Hence, in particular, for $\alpha=\frac{1}{2}$, we obtain
$$Z^{\alpha}_{t}X_{t}=\int_{0}^{t}{Z^{\alpha}_{s}dM_{s}}.$$
 Therefore, $Z^{\alpha}X$ is a local martingale.

$\Leftarrow)$ Now, if we assume that $\exists\alpha\in[0,1]$ such that $Z^{\alpha}X$ is a local martingale, it follows that $Z^{\alpha}X\in(\Sigma)$. Then, it follows from Theorem \ref{z} that $X\in(\Sigma)$.
\end{proof}

It is well known that the absolute value $|M|$ of a continuous local martingale $M$ is an element of the class $(\Sigma)$. In the next corollary, we show that for any stochastic process $X$ of the class $(\Sigma)$, there exists a local martingale $M$ that has the same absolute value as $X$.%Bouhadou and Ouknine have studied the reciprocal problem. They have proved that under the assumption $\alpha=\frac{1}{2}$, any continuous and positive submartingale of $(\Sigma)$, is the absolute value of a martingale  

\begin{coro}\label{pabs}
Let $X$ be a continuous semi-martingale. Then, $X$ is an element of the class $(\Sigma)$ if and only if there exists a local martingale $M$ such that
$$|X|=|M|.$$
\end{coro}
\begin{proof}
$\Rightarrow)$ Assume that $X$ is an element of the class $(\Sigma)$ and define $Z^{\alpha}$ with $\alpha=\frac{1}{2}$. Hence, it follows from Corollary \ref{cmart} that $M=Z^{\alpha}X$ is a continuous local martingale. Then, $|X|=|M|$ since $|Z^{\alpha}X|=|X|$.

$\Leftarrow)$ Now, assume that there exists a continuous martingale $M$ such that $|X|=|M|$. From Tanaka's formula, we get
$$|X_{t}|=|M_{t}|=\int_{0}^{t}{sign(M_{s})dM_{s}}+L_{t}^{0}(M).$$
However, $L_{t}^{0}(M)=L_{t}^{0}(X)$ and $dL_{t}^{0}(X)$ is carried by $\{t\geq0: X_{t}=0\}$. Thus, $|X|\in(\Sigma)$. Consequently, it follows from Theorem \ref{abs} that $X\in(\Sigma)$. 
\end{proof}

{\subsection{Construction of solutions for skew Brownian motion equations}}

This subsection is devoted to the construction of solutions for skew Brownian motion equations. More precisely, we construct solutions from continuous processes of the class $(\Sigma)$ for the following equations:
\begin{equation}\label{1}
	X_{t}=x+B_{t}+(2\alpha-1)L_{t}^{0}(X)
\end{equation}
and
\begin{equation}\label{2}
	X_{t}=x+B_{t}+\int_{0}^{t}{(2\alpha(s)-1)dL_{s}^{0}(X)},
\end{equation}
where $B$ is a standard Brownian motion.

{\subsubsection{Construction of solutions with processes whose martingale part is a Brownian motion}}

 First, we use stochastic processes of the class $(\Sigma)$ whose martingale part is a Brownian motion. Many such processes can be found in the literature. For instance, we have $|B|$, $(\sup_{s\leq t}{B_{s}}-B_{t})_{t\geq0}$ or solutions of equations \eqref{1} and \eqref{2} that start from zero. Our solutions are constructed as follows. For any process $X$ of the class $(\Sigma)$ with a Brownian motion as its martingale part in its Doob--Meyer decomposition, we set $Y^{\alpha}=Z^{\alpha}X$, $|Y^{\alpha}|=Z^{\alpha}|X|$, $\mathcal{Y}^{\alpha}=\mathcal{Z}^{\alpha}X$, and $|\mathcal{Y}^{\alpha}|=\mathcal{Z}^{\alpha}|X|$, where  $Z^{\alpha}$ and $\mathcal{Z}^{\alpha}$ are respectively given in \eqref{zalpha}  and \eqref{Zalpha} and constructed with respect to $X$. These solutions are inspired by the construction of Bouhadou and Ouknine \cite{siam}. In fact, their solution is a particular case of the solutions given in the present study.

When $\alpha$ is a constant, we have the following result.
\begin{prop}
Let $X=M+A$ be a process of the class $(\Sigma)$ such that its martingale part is a Brownian motion. Then, $Y^{\alpha}=Z^{\alpha}X$ and $|Y^{\alpha}|=Z^{\alpha}|X|$ are weak solutions of \eqref{1} with the parameter $\alpha$ and starting from 0.
\end{prop}
\begin{proof}
 From Remark \ref{o}, we have
$$Y^{\alpha}_{t}=W_{t}+(2\alpha-1)L^{0}_{t}(Y^{\alpha})$$
with $W_{t}=\int_{0}^{t}{Z^{\alpha}_{s}dM_{s}}$. Now, let us define a process $k$ as follows:
\begin{equation}\label{k}
	k_{t}=\sum_{n=0}^{+\infty}{\zeta_{n}1_{[g_{n},d_{n}[}(t)}.
\end{equation}
First, we remark that for $\gamma_{t}=\sup\{s\leq t:X_{s}=0\}$, we have
$$k_{\gamma_{t}}X_{t}=Z^{\alpha}_{t}X_{t}.$$
Meanwhile, we obtain the following from Proposition \ref{balpro}:
$$k_{\gamma_{t}}X_{t}={\int_{0}^{t}{^{p}k_{\gamma_{s}}dX_{s}}+R_{t}},$$
where $R$ is an adapted, continuous process with bounded variations such that $dR_{t}$ is carried by the set $\{t\geq0:X_{t}=0\}$. Since $k$ is a càdlàg process, we obtain 
$$k_{\gamma_{t}}X_{t}={\int_{0}^{t}{k_{s-}dX_{s}}+R_{t}}.$$
Finally, from the continuity of $X$, we get 
$$k_{\gamma_{t}}X_{t}={\int_{0}^{t}{k_{s}dX_{s}}+R_{t}}.$$
However,
$$\langle W,W\rangle_{t}=\langle Y^{\alpha},Y^{\alpha}\rangle_{t}=\langle k_{\gamma_{\cdot}}X_{\cdot},k_{\gamma_{\cdot}}X_{\cdot}\rangle_{t}.$$
Then,
$$\langle W,W\rangle_{t}=\int_{0}^{t}{(k_{s})^{2}d\langle X,X\rangle_{s}}=\langle X,X\rangle_{t}=\langle M,M\rangle_{t}$$
since $k_{s}\in\{-1,1\}$. This implies that $\langle W,W\rangle_{t}=t$ because $M$ is a Brownian motion. Thus, $W$ is a Brownian motion. Consequently, $Y^{\alpha}$ is a weak solution of \eqref{1}. Meanwhile, from Theorem \ref{abs}, we have that $|X|$ is a continuous submartingale of the class $(\Sigma)$. Moreover,
$$|X_{t}|=|Z^{\alpha}_{t}X_{t}|=\int_{0}^{t}{{\rm sgn}(Z^{\alpha}_{s}X_{s})d(Z^{\alpha}_{s}X_{s})}+L_{t}^{0}(Z^{\alpha}X),$$
i.e.,
$$|X_{t}|=\int_{0}^{t}{Z^{\alpha}_{s}{\rm sgn}(Z^{\alpha}_{s}X_{s})dM_{s}}+(2\alpha-1)\int_{0}^{t}{{\rm sgn}(Z^{\alpha}_{s}X_{s})dL_{s}^{0}(Z^{\alpha}X)}+L_{t}^{0}(Z^{\alpha}X).$$
Therefore, ${\rm sgn}(Z^{\alpha}_{s}X_{s})=Z^{\alpha}_{s}{\rm sgn}(X_{s})$. Hence,
$$|X_{t}|=\int_{0}^{t}{(Z^{\alpha}_{s})^{2}{\rm sgn}(X_{s})dM_{s}}+(2\alpha-1)\int_{0}^{t}{Z^{\alpha}_{s}{\rm sgn}(X_{s})dL_{s}^{0}(Z^{\alpha}X)}+L_{t}^{0}(Z^{\alpha}X).$$
Thus,
$$|X_{t}|=\int_{0}^{t}{{\rm sgn}(X_{s})dM_{s}}+L_{t}^{0}(Z^{\alpha}X)$$
since $Z^{\alpha}$ is defined on the complementary set of $\{t\geq0:X_{t}=0\}=\{t\geq0:Z^{\alpha}_{t}X_{t}=0\}$ and $dL_{t}^{0}(Z^{\alpha}X)$ is carried by $\{t\geq0:X_{t}=0\}=\{t\geq0:Z^{\alpha}_{t}X_{t}=0\}$. However, we remark that the martingale part $W_{t}=\int_{0}^{t}{{\rm sgn}(X_{s})dM_{s}}$ satisfies the following: $\forall t\geq0$,
$$\langle W,W\rangle_{t}=\int_{0}^{t}{({\rm sgn}(X_{s}))^{2}d\langle M,M\rangle_{s}}=\langle M,M\rangle_{t}=t.$$
In other words, $W$ is a Brownian motion. Consequently, from above, it follows that $|Y^{\alpha}|$ is also a weak solution of \eqref{1}.
   \end{proof}

When $\alpha$ is a piecewise constant, we propose the following solutions.
\begin{prop}\label{p8}
Let $X=M+A$ be a process of the class $(\Sigma)$ such that $M$ is a standard Brownian motion. Then, $\mathcal{Y}^{\alpha}=\mathcal{Z}^{\alpha}X$ and $|\mathcal{Y}^{\alpha}|=\mathcal{Z}^{\alpha}|X|$ are weak solutions of \eqref{2} with the parameter $\alpha$ and starting from 0.
\end{prop}
\begin{proof}
By applying Proposition \ref{pzalph}, we get
$$\mathcal{Y}^{\alpha}_{t}=\int_{0}^{t}{\mathcal{Z}^{\alpha}_{s}dX_{s}}+\int_{0}^{t}{(2\alpha(s)-1)dL_{s}^{0}(\mathcal{Y}^{\alpha})}$$
$$\hspace{2.8cm}=\int_{0}^{t}{\mathcal{Z}^{\alpha}_{s}dM_{s}}+\int_{0}^{t}{\mathcal{Z}^{\alpha}_{s}dV_{s}}+\int_{0}^{t}{(2\alpha(s)-1)dL_{s}^{0}(\mathcal{Y}^{\alpha})}.$$
Hence, we get
$$\mathcal{Y}^{\alpha}_{t}=\int_{0}^{t}{\mathcal{Z}^{\alpha}_{s}dM_{s}}+\int_{0}^{t}{(2\alpha(s)-1)dL_{s}^{0}(\mathcal{Y}^{\alpha})},$$
since $\mathcal{Z}^{\alpha}$ is defined on the complementary set of the zero set of $X$ and $dV_{t}$ is carried by $\{t\geq0:X_{t}=0\}$.
Now, let 
$$k_{t}=\sum_{n=0}^{+\infty}{\sum_{i=0}^{m}{\zeta^{i}_{n}1_{[g_{n},d_{n}[\cap[t-{i},t_{i+1})}(t)}}.$$
We can see that $\forall t\geq0$, $\mathcal{Y}^{\alpha}_{t}=k_{\gamma_{t}}X_{t}$. From Proposition \ref{balpro}, we have
$$k_{\gamma_{t}}X_{t}={\int_{0}^{t}{^{p}k_{\gamma_{s}}dX_{s}}+R_{t}},$$
where $R$ is an adapted, continuous process with bounded variations such that $dR_{t}$ is carried by the set $\{t\geq0:X_{t}=0\}$. Hence, we obtain 
$$k_{\gamma_{t}}X_{t}={\int_{0}^{t}{k_{s-}dX_{s}}+R_{t}},$$
since $k$ is a càdlàg process. Finally, from the continuity of $X$, we get 
$$k_{\gamma_{t}}X_{t}={\int_{0}^{t}{k_{s}dX_{s}}+R_{t}}.$$
Therefore, by letting $W_{t}=\int_{0}^{t}{\mathcal{Z}^{\alpha}_{s}dM_{s}}$, we have
$$\langle W,W\rangle_{t}=\langle \mathcal{Y}^{\alpha},\mathcal{Y}^{\alpha}\rangle_{t}=\langle k_{\gamma_{\cdot}}X_{\cdot},k_{\gamma_{\cdot}}X_{\cdot}\rangle_{t},$$
i.e.,
$$\langle W,W\rangle_{t}=\int_{0}^{t}{(k_{s})^{2}d\langle X,X\rangle_{s}}=\langle X,X\rangle_{t}=\langle M,M\rangle_{t},$$
since $k_{s}\in\{-1,1\}$. This implies that $\langle W,W\rangle_{t}=t$ because $M$ is a Brownian motion. Thus, $W$ is a Brownian motion. Consequently, $\mathcal{Y}^{\alpha}$ is a weak solution of \eqref{2}.

Meanwhile, from above, we can say that $|X|=m+V$ is a continuous process of the class $(\Sigma)$ and its martingale part is a Brownian motion. Consequently, from above, $|\mathcal{Y}^{\alpha}|$ is also a weak solution of \eqref{2}.
\end{proof}

\begin{rem}
Recall that the solution proposed by Bouhadou and Ouknine \cite{siam} is $|\mathcal{Y}^{\alpha}|=\mathcal{Z}^{\alpha}|B|$, where $B$ is a standard Brownian motion and $\mathcal{Z}^{\alpha}$ is constructed relatively at $|B|$. It now clear that this solution is a particular case of the one we presented in Proposition \ref{p8} since $|B|$ is a process of the class $(\Sigma)$ and its martingale part $\left(\int_{0}^{t}{{\rm sgn}(B_{s})dB_{s}}\right)_{t\geq0}$ is a Brownian motion.
\end{rem}

{\subsubsection{Construction of general solutions from the class \texorpdfstring{$(\Sigma)$}{sigma}}

Now, we shall construct solutions of \eqref{1} and \eqref{2} from continuous processes of the class $(\Sigma)$ whose martingale part is not necessarily a Brownian motion. For this purpose, we propose the following constructions. We consider a continuous process $X=M+A$ of the class $(\Sigma)$. We define $\tau_{t}=\inf\{s\geq0:\langle M,M\rangle_{s}>t\}$, $Y_{t}=X_{\tau_{t}}$, and we construct $Z^{\alpha}$ and $\mathcal{Z}^{\alpha}$ with respect to $Y$. Hence, our solutions are constructed as follows: 
$$\forall t\geq0\text{, }Y^{\alpha}_{t}=Z^{\alpha}_{t}Y_{t}\text{, }|Y^{\alpha}_{t}|=Z^{\alpha}_{t}|Y_{t}|\text{, }\mathcal{Y}^{\alpha}_{t}=\mathcal{Z}^{\alpha}_{t}Y_{t}\text{ and }|\mathcal{Y}^{\alpha}_{t}|=\mathcal{Z}^{\alpha}_{t}|Y_{t}|.$$

First, we consider the case of constant $\alpha$.
\begin{prop}
The processes $Y^{\alpha}$ and $|Y^{\alpha}|$ are weak solutions of \eqref{1} with the parameter $\alpha$ and starting from 0.
\end{prop}
\begin{proof}
From Proposition \ref{zalpha}, we have
$$Y^{\alpha}_{t}=\int_{0}^{t}{Z^{\alpha}_{s}dY_{s}^{\alpha}}+(2\alpha-1)L_{t}^{0}(Y^{\alpha}),$$
i.e.,
$$Y^{\alpha}_{t}=\int_{0}^{t}{Z^{\alpha}_{s}dM_{\tau_{s}}}+\int_{0}^{t}{Z^{\alpha}_{s}dA_{\tau_{s}}}+(2\alpha-1)L_{t}^{0}(Y^{\alpha}).$$
Hence,
$$Y^{\alpha}_{t}=\int_{0}^{t}{Z^{\alpha}_{s}dM_{\tau_{s}}}+(2\alpha-1)L_{t}^{0}(Y^{\alpha})$$
since $dA_{\tau_{s}}$ is carried by $\{s\geq0:X_{\tau_{s}}=0\}=\{s\geq0:Z^{\alpha}_{s}=0\}$.
Then,
$$Y^{\alpha}_{t}=\int_{0}^{\tau_{t}}{Z^{\alpha}_{\langle M,M\rangle_{s}}dM_{s}}+(2\alpha-1)L_{t}^{0}(Y^{\alpha}),$$
i.e.,
$$Y^{\alpha}_{t}=W_{t}+(2\alpha-1)L_{t}^{0}(Y^{\alpha}),$$
with $W_{t}=\int_{0}^{\tau_{t}}{Z^{\alpha}_{\langle M,M\rangle_{s}}dM_{s}}$.

Now, we construct the process $k$ with respect to $Y$. More precisely, 
$$k_{t}=\sum_{n=0}^{+\infty}{\zeta_{n}1_{[g_{n},d_{n}[}(t)}.$$
 We can see that $\forall t\geq0$, $Y^{\alpha}_{t}=k_{\gamma_{t}}Y_{t}$. From Proposition \ref{balpro}, we have
$$k_{\gamma_{t}}Y_{t}={\int_{0}^{t}{^{p}k_{\gamma_{s}}dY_{s}}+R_{t}},$$
where $R$ is an adapted, continuous process with bounded variations such that $dR_{t}$ is carried by the set $\{t\geq0:Y_{t}=0\}$. Hence, we obtain 
$$k_{\gamma_{t}}Y_{t}={\int_{0}^{t}{k_{s-}dY_{s}}+R_{t}},$$
since $k$ is a càdlàg process. It follows from the continuity of $Y$ that  
$$k_{\gamma_{t}}Y_{t}={\int_{0}^{t}{k_{s}dY_{s}}+R_{t}}.$$
Therefore, we have
$$\langle W,W\rangle_{t}=\langle Y^{\alpha},Y^{\alpha}\rangle_{t}=\langle k_{\gamma_{\cdot}}Y_{\cdot},k_{\gamma_{\cdot}}Y_{\cdot}\rangle_{t},$$
i.e.,
$$\langle W,W\rangle_{t}=\int_{0}^{\tau_{t}}{(k_{\langle M,M\rangle_{s}})^{2}d\langle X,X\rangle_{s}}=\langle X,X\rangle_{\tau_{t}}=\langle M,M\rangle_{\tau_{t}}$$,
since $k_{\langle M,M\rangle_{s}}\in\{-1,1\}$. This implies that $\langle W,W\rangle_{t}=t$. Thus, $W$ is a Brownian motion. Consequently, $Y^{\alpha}$ is a weak solution of \eqref{1} with the parameter $\alpha$ and starting from 0. 

Meanwhile, from Theorem \ref{abs}, we have $|X|\in(\Sigma)$. Then, from above, $|Y^{\alpha}|$ is also a weak solution of \eqref{1} with the parameter $\alpha$ and starting from 0.
\end{proof}

Now, we propose solutions for \eqref{2} in the following proposition.

\begin{prop}
$\mathcal{Y}^{\alpha}_{t}$ and $|\mathcal{Y}^{\alpha}_{t}|$ are weak solutions of \eqref{2} with the parameter $\alpha$ and starting from 0.
\end{prop}
\begin{proof}
From Proposition \ref{pzalph}, we have
$$\mathcal{Y}^{\alpha}_{t}=\int_{0}^{t}{\mathcal{Z}^{\alpha}_{s}dY_{s}}+\int_{0}^{t}{(2\alpha(s)-1)dL_{s}^{0}(\mathcal{Y}^{\alpha})}$$
$$\hspace{2.8cm}=\int_{0}^{t}{\mathcal{Z}^{\alpha}_{s}dM_{\tau_{s}}}+\int_{0}^{t}{\mathcal{Z}^{\alpha}_{s}dV_{\tau_{s}}}+\int_{0}^{t}{(2\alpha(s)-1)dL_{s}^{0}(\mathcal{Y}^{\alpha})}.$$
This implies that
$$\mathcal{Y}^{\alpha}_{t}=\int_{0}^{t}{\mathcal{Z}^{\alpha}_{s}dM_{\tau_{s}}}+\int_{0}^{t}{(2\alpha(s)-1)dL_{s}^{0}(\mathcal{Y}^{\alpha})},$$
since $\mathcal{Z}^{\alpha}$ is defined on the complementary set of the zero set of $Y$ and $dV_{\tau_{t}}$ is carried by $\{t\geq0:Y_{t}=0\}$.
Now, we consider the following process $k$: 
$$k_{t}=\sum_{n=0}^{+\infty}{\sum_{i=0}^{m}{\zeta^{i}_{n}1_{[g_{n},d_{n}[\cap[t-{i},t_{i+1})}(t)}},$$
which is defined with respect to $Y$. We can see that $\forall t\geq0$, $\mathcal{Y}^{\alpha}_{t}=k_{\gamma_{t}}Y_{t}$. From Proposition \ref{balpro}, we have
$$k_{\gamma_{t}}Y_{t}={\int_{0}^{t}{^{p}k_{\gamma_{s}}dY_{s}}+R_{t}},$$
where $R$ is an adapted, continuous process with bounded variations such that $dR_{t}$ is carried by the set $\{t\geq0:Y_{t}=0\}$. Hence, we obtain 
$$k_{\gamma_{t}}Y_{t}={\int_{0}^{t}{k_{s-}dY_{s}}+R_{t}},$$
since $k$ is a càdlàg process. It follows from the continuity of $Y$ that  
$$k_{\gamma_{t}}Y_{t}={\int_{0}^{t}{k_{s}dY_{s}}+R_{t}}.$$
Therefore, by letting $W_{t}=\int_{0}^{t}{\mathcal{Z}^{\alpha}_{s}dM_{\tau_{s}}}$, we have
$$\langle W,W\rangle_{t}=\langle \mathcal{Y}^{\alpha},\mathcal{Y}^{\alpha}\rangle_{t}=\langle k_{\gamma_{\cdot}}Y_{\cdot},k_{\gamma_{\cdot}}Y_{\cdot}\rangle_{t},$$
i.e.,
$$\langle W,W\rangle_{t}=\int_{0}^{\tau_{t}}{(k_{\langle M,M\rangle_{s}})^{2}d\langle X,X\rangle_{s}}=\langle X,X\rangle_{\tau_{t}}=\langle M,M\rangle_{\tau_{t}},$$
since $k_{\langle M,M\rangle_{s}}\in\{-1,1\}$. This implies that $\langle W,W\rangle_{t}=t$. Thus, $W$ is a Brownian motion. Consequently, $\mathcal{Y}^{\alpha}$ is a weak solution of \eqref{2}.

Meanwhile, recall from Theorem \ref{abs} that $|X|$ is a continuous semi-martingale of the class $(\Sigma)$. Therefore, from above, we can say that $|\mathcal{Y}^{\alpha}|$ is also a weak solution of \eqref{2}.
\end{proof}

{\color{myaqua}


\begin{thebibliography}{10}
{\color{black}

%\bibitem{Akdim}
%K.~Akdim, M.~Eddahbi and M.~Haddadi.
%\newblock Characterization of submartingales of a new class \texorpdfstring{$(\Sigma^{r})$}{}. \newblock {\em Stochastic Analysis and Applications}, DOI:10.1080/07362994.2018.1429932.

\bibitem{1}
J.~Azéma and M.~Yor.
\newblock Sur les zéros des martingales continues. \newblock {\em Séminaire de probabilités (Strasbourg)}, 26: 248-306, 1992.
\bibitem{2}
J.~Azéma and M.~Yor.
\newblock Une solution simple au problème de Skorokhod. \newblock {\em in: Sém.proba. XIII, in: Lecture Notes in
Mathematics}, 721: 90-115,625-633, 1979.

\bibitem{siam}
S.~Bouhadou and Y.~Ouknine.
\newblock On the time inhomogeneous skew Brownian motion. \newblock {\em Bulletin des Sciences Mathématiques}, vol.137(7): 835-850, 2013.
%\bibitem{sak}
%S.~Beghdadi-Sakrani.
%\newblock Calcul stochastique pour les mesures signées. \newblock {\em Séminaire de probabilités (Strasbourg)}, 36: 366-382, 2002.

%\bibitem{AY3}
%L.~Carraro, N.~El Karoui and J.~ Obl\'oj.
%\newblock On Azema-Yor Processes, their Optimal Properties and the Bachelier-Drawdown Equation. \newblock {\em Annals of Probability}, 40(1): 372-400, 2012

\bibitem{pat}
P.~Cheridito, A.~Nikeghbali and E.~Platen.
\newblock Processes of class sigma, last passage times, and drawdowns. \newblock {\em SIAM Journal on Financial Mathematics}, 3(1): 280-303, 2012

\bibitem{pot}
C.~Dellacherie, P.A.~Meyer.
\newblock Probabilités et Potentiel. \newblock {\em Chapitres V à VIII. Théorie des Martingales. Revised Edition, Hermann, 1980, Paris}
%\bibitem{f}
%F.~Eyi-Obiang, Y.~Ouknine and O.~Moutsinga.
%\newblock New classes of processes in Stochastic calculus for signed measures. \newblock {\em Stochastics}, 86(1): 70-86, 2014.

%\bibitem{e}
%F.~Eyi-Obiang, Y.~Ouknine and O.~Moutsinga.
%\newblock On the study of processes of classes  \texorpdfstring{$\Sigma(H)$}{} and \texorpdfstring{$\Sigma_{s}(H)$}{}. \newblock {\em Journal of Theoretical Probability}, (2017)30:117-142. DOI:10.1007/s10959-015-0640-x

\bibitem{el}
N.~El Karoui.
\newblock Sur les montées des semi-martingales. \newblock {\em Astérisque}, Tome 52-53,p. 63-72, 1978.

\bibitem{elk}
N.~El Karoui.
\newblock Sur les montées des semi-martingales le cas non continue. \newblock {\em Astérisque}, Tome 52-53,p. 73-87, 1978.

\bibitem{9}
P.~Etoré and M.~Martinez.
\newblock on the existence of a time inhomogeneous skew Brownian motion and some related laws. \newblock {\em Electron. J. Probab}, 17(19): 1-27, 2012. 

\bibitem{eomt}
F.~Eyi-Obiang, Y.~Ouknine and O.~Moutsinga, G.~Trutnau.
\newblock Some contributions to the study of stochastic processes of the classes  \texorpdfstring{$\Sigma(H)$}{} and \texorpdfstring{$(\Sigma)$}{}. \newblock {\em Stochastics, 89, 8: 1253-1269, 2017.}

\bibitem{gila}
D.~Gila.
\newblock Every non-negative submartingale is the absolute value of a martingale. \newblock {\em Ann.Probab}, 5, p. 475-481, 1977.

\bibitem{10}
J.M.~Harrison and L.A.~Shepp.
\newblock On skew Brownian motion. \newblock {\em Ann.Probab}, 9, no. 2, 309-313, 1981.

%\bibitem{6}
%T.~Jeulin and M.~Yor.
%\newblock Sur les distributions de certaines fonctionnelles du mouvement brownien. \newblock {\em in: Sém.proba. XV, in: Lecture Notes in
%Mathematics}, vol. 850: 210-226, 1981.
\bibitem{11}
K.~Itô and H.P.~McKean.
\newblock Diffusion and Their Sample Paths. \newblock {\em Springer-Verlag $2^{nd}$ edition}, 1974.

\bibitem{12}
J.F.~Le Gall.
\newblock One-dimensional stochastic differential equations involving local times of unknown process. \newblock {\em Stochastic analysis and application (Swansea 1983) 51-82} Lecture notes in mathematics 1095, Springer-verlag, Berlin 1984.

%\bibitem{levy}
%P.~Lévy.
%\newblock Sur les montées des semi-martingales (cas continu). \newblock {\em Astérisque}, 52-53,S.M.F 1978.

\bibitem{man}
R.~Mansury, M.~Yor.
\newblock Random Times and Enlargements of Filtrations in a Brownian Setting. \newblock {\em Lecture Notes in Mathematics 1873, Springer, 2006, ISBN 3540294074, DOI 10.1007/11415558}

\bibitem{MSY}
P.A.~Meyer, C.~Stricker, M.~Yor.
\newblock Sur une formule de la théorie du balayage. \newblock {\em Séminaire de probabilités (Strasbourg),} 13: 478-487, 1979

\bibitem{naj}
J.~Najnudel, A.~Nikeghbali.
\newblock A new construction of the $\Sigma$- finite measures associated with submartingales of class $(\Sigma)$. \newblock {\em C.R. Math. Acad. Sci. Paris}, 348: 311-316, 2010.
\bibitem{naj1}
J.~Najnudel, A.~Nikeghbali.
\newblock A remarkable sigma-finite measure associated with last passage times and penalisation results. \newblock {\em Contemporary Quantitative Finance, Essays in Honour of Eckhard Platen,Springer}, 77-98, 2010.


\bibitem{naj2}
J.~Najnudel, A.~Nikeghbali.
\newblock On some properties of a universal sigma finite measure associated with a remarkable class of submartingales. \newblock {\em Publ. of the Res. Instit. for Math. Sci. (Kyoto University)}, 47(4): 911-936, 2011.

\bibitem{naj3}
J.~Najnudel, A.~Nikeghbali.
\newblock On some universal sigma-finite measures and some extensions of Doob's optional stopping theorem. \newblock {\em Accepted in Stochastic processes and their applications}.

\bibitem{nik}
A.~Nikeghbali.
\newblock A class of remarkable submartingales. \newblock {\em Journal of Theoretical Probability}, 4(19): 931-949, 2006.
\bibitem{mult}
A.~Nikeghbali.
\newblock Multiplicative decompositions and frequency of vanishing of nonnegative submartingales. \newblock {\em Journal of Theorical Probability}, 19(4): 931-949, 2006.

%\bibitem{doob}
%A.~Nikeghbali and M.~Yor.
%\newblock Doob's maximal identity, multiplicative decompositions and enlargements of filtrations. \newblock {\em Illinois J. Math.(in press)}, 2005.

%\bibitem{obloj}
%J.~Obl\'oj.
%\newblock The Skorokhod embedding problem and its offspring. \newblock {\em Probab. Surv.}, 1:321-390, 2004.

%\bibitem{15}
%J.~Obl\'oj and M.~Yor.
%\newblock An explicit Skorokhod embedding for the age of Brownian excursions and Azéma martingale. \newblock {\em Stochastic Process. Appl.}, 110(1):83-110, 2004.

%\bibitem{chav}
%J.~Ruiz de Chavez.
%\newblock Le théorème de Paul Lévy pour des mesures signées. \newblock {\em Séminaire de probabilités (Strasbourg)}, 18: 245-255, 1984.

%\bibitem{skorokhod}
%A.V.~Skorokhod.
%\newblock Studies in the Theory of Random Processes (Trans. Scripta Technica Inc.,). \newblock {\em Addison-Wesley
%Publishing Co., Inc., Reading, Mass., 1965 (Original work published in Russian).}
\bibitem{16}
Y.~Ouknine.
\newblock ''Skew-Brownian motion'' and derived processes. \newblock {\em Theory Probab. Appl.}, 35: 163-169, 1990.

%\bibitem{prok}
%V.~Prokaj.
%\newblock Unfolding the Skorokhod reflection of a semi-martingale. \newblock {\em Statistics and probability Letters}, 79:534-536, 2009.

\bibitem{21}
J.B.~Walsh.
\newblock A diffusion with a discontinuous local time. In Temps locaux, Astérisques, pp. 37-45 \newblock {\em Société Mathématique de France}, 1978.

\bibitem{y1}
M.~Yor.
\newblock Les inégalités de sous-martingales, comme conséquences de la relation de domination. \newblock {\em Stochastics}, 3(1): 1-15, 1979.

\bibitem{y}
M.~Yor.
\newblock Sur le balayage des semi-martingales continues. \newblock {\em Séminaire de probabilités (Strasbourg)}, 13: 453-471, 1979.

}

\end{thebibliography}
\end{document}